\documentclass[12pt]  {amsart} 
\usepackage{amsmath,latexsym,amssymb,amsmath,
amscd,amsthm,amsxtra}
\usepackage{txfonts} 

\newtheorem{definition}{Definition}[section]
\newtheorem{theorem}{Theorem}[section]

\newtheorem{lemma}{Lemma}[section]
\newtheorem{corollary}{Corollary}[section]

\def\l{\lambda}
\def\L{\Lambda}

\def\<{\langle}
\def\>{\rangle}
\def\div{{\rm div}}

\def \ds{\displaystyle}
\def \vs{\vspace*{0.1cm}}

\begin{document}

\title{Local gradient estimate for harmonic functions on Finsler manifolds}

\author{Chao Xia}\address{Max-Planck-Institut f\"ur Mathematik in den Naturwissenschaften, Inselstr. 22, D-04103, Leipzig, Germany}
\email{chao.xia@mis.mpg.de}\thanks{The research leading to these results has received funding from the European Research Council
under the European Union's Seventh Framework Programme (FP7/2007-2013) / ERC grant
agreement no. 267087.}
\maketitle

\begin{abstract}{In this paper, we prove the local gradient estimate for harmonic functions on complete, noncompact Finsler measure spaces under the condition that the weighted Ricci curvature has a lower bound. As applications, we obtain Liouville type theorems on noncompact Finsler manifolds with nonnegative Ricci curvature.}

\end{abstract}

\section{Introduction}
The study of harmonic functions is one of the center topics in geometric analysis. In a seminal paper \cite{Y}, Yau proved that complete Riemannian manifolds with nonnegative Ricci curvature must have Liouville property. He derived a gradient estimate which is nowadays important and essential to the theory of harmonic functions. Later, Cheng-Yau \cite{CY} proved the following local version of Yau's gradient estimate.

\

\noindent{\bf Theorem A} (\cite{Y,CY}). {\it  Let $M^n$ be an n-dimensional complete noncompact
Riemannian manifold with $Ricci\geq -K$ $ (K \geq 0)$. Then
there exists a constant $C_n$ depending only on n, such that every positive harmonic function $u$
on geodesic ball $B_{2R}(p)\subset M$  satisfies
\begin{eqnarray*}
|\nabla  \log u|\leq C_n\frac{1+\sqrt{K}R}{R}  \hbox{ in }B_R(p).
\end{eqnarray*}
}

The main objective of this paper is to generalize the local gradient estimate to Finsler manifolds.
 
A Finsler manifold $M$ is a differential  manifold whose tangent spaces are equipped with Minkowski norms $F$, but not necessary inner products. A Finsler manifold equipped with a measure $m$ is called a Finsler measure space $(M,F,m)$. The class of Finsler measure spaces is one of the most important metric measure spaces.  There is a canonical gradient $\nabla u$ and Finsler-Laplacian $\Delta_m u$ of a function $u$, defined on a Finsler manifold. Also the harmonic functions on Finsler manifolds are  solutions of $\Delta_m u=0$. It is easy to see that the harmonic functions are the local minimizers of the energy functional $E(u)=\int_M F^2(x,\nabla u) dm$. Note that this energy functional coincides with Cheeger's one \cite{Ch, OS1} in terms of upper gradients for general metric measure spaces. Unlike the Laplacian on Riemannian manifolds, the Finsler-Laplacian  is a nonlinear operator. In fact,  this is the major difference between Finsler and other metric measure spaces (See \cite{ AGS, OS1}). 

On the other hand, the flag and Ricci curvatures are well defined on Finsler manifolds. As in Riemannian manifolds, they describe implicitly global properties of Finsler manifolds. On  Finsler measure spaces, the weighted Ricci curvature $Ric_N$ for $N\in [n,\infty]$ was introduced by Ohta \cite{Oh1}. He proved that the condition that  $Ric_N$ has a lower bound is equivalent to the curvature-dimension condition,  introduced by Lott-Villani \cite{LV} and Sturm \cite{St, St2}, for a Finsler measure space as a metric measure space.  We refer to Section 2, or  \cite{Oh1, OS1} for details of definitions.

Under the assumption that $Ric_N$ has a lower bound, we are able to generalize the local gradient estimate for harmonic functions on Riemannian manifolds to that for harmonic functions on  Finsler measure spaces whose Finsler structures $F$ are \textit{uniformly smooth} and \textit{uniformly convex}. The uniform smoothness and the uniform convexity mean that there exist two uniform constants $0<\l\leq\L<\infty$ such that for $x\in M,$ $V\in T_x M\setminus\{0\}$ and  $W\in T_xM$, we have
\begin{eqnarray}\label{unif} 
\l {F}^2(x, W)\leq \sum_{i,j=1}^n g_V(W,W)\leq\Lambda {F}^2(x,W),
\end{eqnarray}
where $g_V$ is the induced metric on the tangent bundle of corresponding Finsler manifolds, see \eqref{gV} in Section 2. 

Our main result is the following theorem.
\begin{theorem}\label{main thm}
Let $(M^n, F,m)$ be an $n$-dimensional forward geodesically complete, noncompact Finsler measure space,  equipped with a uniformly smooth and uniformly convex Finsler structure $F$ and a smooth measure $m$. 
  Assume that $Ric_N\geq -K$ for some real numbers $N\in [n,+\infty)$ and $K\geq 0$. Let $u$ be a positive harmonic function, i.e. $$\Delta_m u=0$$ in weak sense in a forward geodesic ball $B^+_{2R}(p)\subset M$. Then there exists some constant $C=C(N,\l,\L)$, depending on $N$, the uniform  constants $\l$ and  $\L$ in \eqref{unif},  such that
 \begin{eqnarray*}
\max\left\{ F(x,\nabla\log u(x)),F(x,\nabla (-\log u(x)))\right\}\leq C\frac{1+\sqrt{K}R}{R}\hbox{ in }B^+_R(p).
 \end{eqnarray*}
\end{theorem}

\

The precise definitions of the Finsler measure spaces, forward geodesical completeness, weighted Ricci curvature $Ric_N$, gradient 
$\nabla$, Finsler-Laplacian $\Delta_m$ will be given in Section 2 below. We give several remarks on Theorem \ref{main thm}.

\

\noindent\textbf{Remarks.}
\begin{itemize}
\item[(i)] Theorem \ref{main thm} does not coincide with the local gradient estimate on weighted Riemannian manifold $(M, g_{\nabla u},m)$ since $Ric_N$ and the weighted Ricci curvature $Ric_N^{\nabla u}$ of $(M, g_{\nabla u},m)$ are different, as observed in \cite{OS2}. In fact, the weighted Ricci curvature $Ric_N$ introduced by Ohta depends only on the Finsler structure $F$ and the measure $m$, while the weighted Ricci curvature $Ric_N^{\nabla u}$ depends on $u$. A simple example  is $(\mathbb{R}^n, \|\cdot\|, m_{BH})$, where $\|\cdot\|$ is a Minkowski norm and $m_{BH}$ denotes the Busemann-Hausdorff measure. In this case, $Ric_n$ vanishes but $Ric_n^{\nabla u}$ does not.
\item[(ii)] For $N=n$, $Ric_n\geq -K$ means that Shen's $S$-curvature vanishes and $Ric\geq -K$. Finsler manifolds of \textit{Berwald type}  equipped with \textit{Busemann-Hausdorff measure} satisfies that Shen's $S$-curvature vanishes. In general, there may not exist any measure with vanishing $S$-curvature. See \cite{Sh, Oh1, Oh3}.
\item[(iii)] Ohta-Sturm \cite{OS2} proved the Li-Yau type gradient estimate for heat flow $u_t=\Delta_m u$ on compact Finsler manifolds. They proposed the question for noncompact Finsler manifolds. Theorem \ref{main thm} can be viewed as a special case, i.e., $u_t=0$. However, the Li-Yau type gradient estimate for noncompact Finsler manifolds is still unknown.
\item[(iv)] The uniform smoothness and the uniform convexity was first introduced in Banach space theory by Ball, Carlen and Lieb in \cite{BCL}. In \cite{Oh2}, Ohta gave a geometric interpretation of these two conditions for metric spaces (see (1.3) and (1.4) in \cite{Oh2}) and he also proved that a large class of Finsler manifolds satisfies them (see Theorem 4.2 and Theorem 5.1 in \cite{Oh2}).
\end{itemize}

\

In the Riemannian case, Cheng-Yau \cite{CY} combined the Bochner technique and maximum principle to prove the local gradient estimate. This method turned out to be very useful to many related problem, such as eigenvalue estimate, heat kernel estimate and so on. It is natural to use the same method to handle the Finsler case. For compact Finsler manifolds, it is possible to use the maximum principle since the Bochner formula for $F^2(x,\nabla u)$ is similar as in Riemannian case, see \cite{OS2}. For noncompact Finsler manifolds,  one should compute the equation for $\eta F^2(x,\nabla u)$ and use the maximum principle, where $\eta$ is some cut-off function. However, some unexpected term appears, mainly because of the nonlinearity of the Finsler-Laplacian. Thus the classical method of maximum principle does not work  for noncompact Finsler manifolds any more.

Fortunately, we find that the method of Moser's iteration is effective to avoid the difficulty caused by the nonlinearity of the Finsler-Laplacian. This is inspired by recent work by Wang-Zhang \cite{WZ} on similar result for $p$-Lapalcian on Riemannian manifolds. We will start from a Bochner formula established recently by Ohta-Sturm \cite{OS2} and ultilize carefully  Moser's iteration to prove Theorem \ref{main thm}.

Recently, the local gradient estimate has been extended to Alexandrov spaces by Zhang-Zhu \cite{ZZ}, see also \cite{HX}.

\

Several standard applications of Theorem \ref{main thm} are the Harnack inequality and the Liouville properties.
\begin{corollary}\label{11}
 Let $(M^n,F,m)$ be as in Theorem \ref{main thm} and  $u$ be a positive harmonic function in geodesic ball $B_{2R}^+(p)\subset M$. Then there exists some constant $C=C(N,\l,\L)$, depending on $N$, $\l$ and $\L$, such that
 \begin{eqnarray*}
\sup_{B_R^+(p)}u\leq e^{C(1+\sqrt{K}R)}\inf_{B_R^+(p)} u
 \end{eqnarray*}
\end{corollary}

\begin{corollary}\label{12}
 Let $(M^n,F,m)$ be as in Theorem \ref{main thm} with $K=0$. Then any positive harmonic function on $M$ must be a  constant.
 \end{corollary}
 
 \begin{corollary}\label{13}
 Let $(M^n,F,m)$ be as in Theorem \ref{main thm} with $K=0$. Then any harmonic function $u$ satisfying $$\lim_{d(p,x)\to \infty} \frac{|u(x)|}{d(p,x)}=0\hbox{ for some }p\in M$$  must be a constant.\end{corollary}

 In the remaining part of this paper, we will first recall the basis of Finsler manifolds. Then we prepare the analytic tool, the Sobolev
inequality, for  Moser's iteration. Finally we prove
Theorem \ref{main thm} and the corollaries.

\

 \section{Preliminaries on Finsler geometry}

In this section  we briefly recall the fundamentals of Finsler geometry,
 as well as  the recent developments on the analysis of Finsler geometry by Ohta-Sturm \cite{Oh1, OS1, OS2}.
For Finsler geometry, we refer to \cite{BCS, Sh}.

\subsection{Finsler measure spaces}

Let $M^n$ be a smooth, connected $n$-dimensional  manifold. A function $F:TM\to[0,\infty)$ is called a \textit{Finsler structure} (or \textit{Minkowski norm}) if it satisfies the following properties:
\begin{itemize}
\item[(i)] $F$ is $C^\infty$ on $TM\setminus\{0\}$;
\item [(ii)] $F(x, tV) = t F(x, V)$ for all $(x, V)\in TM$ and all $t > 0$;
\item [(iii)] for every $(x, V)\in TM\setminus\{0\}$, the matrix
\begin{eqnarray*}
g_{ij}(x,V):=\frac{\partial^2}{\partial V_i\partial V_j}(\frac12 F^2)(x, V)
\end{eqnarray*}
 is positive definite.
\end{itemize}
Such a pair $(M^n,F)$ is called a \textit{Finsler manifold}. By a {\it Finsler measure space}  we mean  a triple $(M^n,F,m)$ constituted with a smooth, connected $n$-dimensional manifold $M$, a Finsler structure $F$ on $M$ and a measure $m$ on $M$. For every non-vanishing vector field $V$, $g_{ij}(x,V)$ induces a Riemannian structure $g_V$ on $M$  via
\begin{eqnarray}\label{gV}
g_V(X,Y)=\sum_{i,j=1}^n g_{ij}(x,V)X^iY^j, \hbox{ for }X,Y\in T_xM.
\end{eqnarray}
In particular, $g_V(V,V)=F^2(x,V)$.

A Finsler structure is said to be \textit{reversible} if, in addition, $F$ is even. Otherwise $F$
is non-reversible. 
One can define the \textit{reversed norm} $\bar{F}(x,V):=F(x,-V)$ and the modulus of reversibility \begin{eqnarray}\label{reverse}
\rho:=\sup_{V\neq 0}\frac{F(x,V)}{\bar{F}(x,V)}.
\end{eqnarray}
 It is easy to see that  $\rho\equiv 1$ for reversible $F$. In general, under the assumption of uniform smoothness and convexity, $\rho$ can be estimated by $\l$ or $\L$ by using \eqref{unif}:
\begin{eqnarray}\label{rho}
&&F^2(x,V)\leq \frac{1}{\l}g_{ij}(x,-V)V^iV^j=\frac{1}{\l}F^2(x,-V),\\&&F^2(x,-V)\geq \frac{1}{\L}g_{ij}(x,V)(-V^i)(-V^j)=\frac{1}{\L}F^2(x,V).
\end{eqnarray}

For $p,q\in M$, the \textit{distance function} from $p$ to $q$ is defined by
$$d_p(q):=d(p,q):=\inf_\gamma \int_0^1 F(\gamma(t),\dot{\gamma}(t))dt,$$
where the infimum is taken over all $C^1$-curves $\gamma:[0,1]\to M$ such that $\gamma(0)=p$ and $\gamma(1)=q$. Note that the distance function may not be symmetric unless $F$ is reversible.
A $C^\infty$-curve $\gamma:[0,1]\to M$ is called a \textit{geodesic} (of constant speed) if $F(\gamma,\dot{\gamma})$ is constant and it is locally minimizing.
The \textit{forward geodesic balls} are defined by
$$B^+_R(p):=\{q\in M: d(p,q)<R\}.$$ 
The \textit{exponential map} $\exp_p :T_pM\to M$ is defined by $\exp_p(v)=\gamma(1)$ for $v\in T_pM$ if there is a  geodesic $\gamma:[0,1]\to M$ with $\dot{\gamma}(0)=v$. A Finsler manifold $(M,F)$ is said to be \textit{forward geodesically complete} if the exponential map is defined on the entire $TM$. By Hopf-Rinow theorem (see \cite{BCS}), any two points in $M$ can be connected by a minimal forward geodesic and the forward closed balls $\overline{B_R^+(p)}$ are compact.
For a point $p\in M$ and a unit vector $v\in T_pM$, let $t_0=\sup\{t>0|\hbox{ the geodesic }\exp_p(tv)\hbox{ is minimal}\}$. If $t_0<\infty$, we call $\exp_p(t_0v)$ a \textit{cut point} of $x$. All the cut points of $x$ is said to be the \textit{cut locus} of $p$.  The cut locus of $p$ always has null measure and $d_p$ is $C^1$ outside the cut locus of $p$ (see \cite{BCS}).


There exists a unique linear connection, which is called the \textit{Chern connection}, on Finsler manifolds. The Chern connection is determined by the following structure equations, which characterize ``torsion freeness":
\begin{eqnarray*}\label{torsion free}
D_X^V Y-D_Y^V X=[X,Y]\end{eqnarray*}
 and ``almost $g$-compatibility"
\begin{eqnarray*}\label{compatible}
Z(g_V(X,Y))=g_V(D_Z^V X,Y)+g_V(X,D_Z^V Y)+C_V(D_Z^V V,X,Y)\end{eqnarray*}
for $V\in TM\setminus\{0\}, X,Y,Z\in TM$. Here $$C_V(X,Y,Z):=C_{ijk}(V)X^i Y^j Z^k=\frac{1}{4}\frac{\partial^3 F^2}{\partial V^i V^j V^k}(\cdot,V)X^i Y^j Z^k$$ denotes the \textit{Cartan tensor}. 

\vspace{2mm}
\subsection{Gradient, Hessian and Finsler-Laplacian}

We shall introduce the Finsler-Laplacian on Finsler measure spaces.

Given a Finsler structure $F$ on $M$, there is a natural dual norm $F^*$ on the cotangent bundle $T^*M$, which is defined by
$$F^*(x,\xi):=\sup_{F(x, V)\leq 1} \xi(V) \hbox{ for any }\xi\in T_x^*M.$$
One can show that  $F^*$ is also a Minkowski norm on $T^*M$ and
$$g_{ij}^*(x,\xi):=\frac{\partial^2}{\partial \xi_i\partial \xi_j}(\frac12 F^{*2})(x, \xi)$$ is positive definite for every $(x, \xi)\in T^*M\setminus\{0\}$.

The \textit{Legendre transform} is defined by the map $l:T_xM\to T_x^*M:$
\begin{equation*}l(V):=\left\{
\begin{array}{lll} g_V(V,\cdot)&\hbox{ for }V\in T_xM\setminus\{0\},\\
0 &\hbox{ for }V=0.\\
\end{array}
\right.
\end{equation*}
Denote by $g^{ij}(x, V)$ the inverse matrix of $g_{ij}(x, V)$.
One can verify that $$F(V)=F^*(l(V)), \hbox{ for all } V\in TM\hbox{ and }g_{ij}^*(x,l(V))=g^{ij}(x, V).$$
From the uniform smoothness and convexity \eqref{unif} one easily see that $g^{ij}$ is  uniform elliptic in the sense that there exists two constants $0<\tilde{\l}\leq\tilde{\L}<\infty$, depending on $\l$ and $\L$, such that for $x\in M$, $V\in T_x M\setminus\{0\}$ and $\xi\in T_x^*M$, we have
\begin{eqnarray}\label{unif1} 
\tilde{\l} {F^*}^2(x,\xi)\leq \sum_{i,j=1}^n g^{ij}(x,V)\xi_i\xi_j\leq\tilde{\Lambda} {F^*}^2(x,\xi). \end{eqnarray}
 
Let $u:M\to \mathbb{R}$ be a smooth function on $M$ and $Du$ be its differential $1$-form. The \textit{gradient} of $u$ is defined as $\nabla u(x):=l^{-1}(Du(x))\in T_xM$. Denote $M_u:=\{Du\neq 0\}$. Locally we can write in coordinates
\begin{equation*}
\nabla u=\sum_{i,j=1}^n g^{ij}(x,\nabla u)\frac{\partial u}{\partial x_i}\frac{\partial}{\partial x_j} \hbox{ in } M_u.
\end{equation*}
Notice that $g^{ij}(x,\nabla u)=g_{ij}^*(x,Du)$ and $F^2(x,\nabla u)={F^*}^2(x,Du)$.

The \textit{Hessian} of $u$ is defined by using Chern connection as
\begin{eqnarray*}\label{hess sym}
\nabla^2 u(X,Y)=g_{\nabla u}(D_X^{\nabla u} \nabla u, Y),
\end{eqnarray*}
One can  show that $\nabla^2 u(X,Y)$ is symmetric, see \cite{OS2, WX}. 

 Let $V\in TM$ be a smooth vector field on $M$. The \textit{divergence} of $V$ with respect to $m$ is defined by
\begin{eqnarray*}
\div_m V dm=d(V\lrcorner dm),
\end{eqnarray*}
where $V\lrcorner dm$ is the $(n-1)$ form given by $$V\lrcorner dm(W_1,\cdots,W_{n-1})=dm(V,W_1,\cdots,W_{n-1}),\quad W_i\in TM, \quad i=1,\cdots, n-1.$$
In local coordinates $(x^i)$, expressing $dm=e^\Phi dx^1 dx^2\cdots dx^n$, we can write $\div_m V$ as
\begin{eqnarray*}
\div_m V=\sum_{i=1}^n\left(\frac{\partial V^i}{\partial x^i}+V^i\frac{\partial \Phi}{\partial x^i}\right).
\end{eqnarray*}

The \textit{Finsler-Laplacian},  can now be defined by
\begin{eqnarray*}
\Delta_m u:=\div_m(\nabla u).
\end{eqnarray*}

Recall that the classes $L^2_{\rm{loc}}(M)$ and $W^{1,2}_{\rm{loc}}(M)$ are defined only in terms of the manifold structure of $M$ (independent of $F$ and $m$). Let 
$$W^{1,2}(M)=\left\{u\in W^{1,2}_{\rm{loc}}(M)\bigcap L^2(M): \int_M F^2(x,\nabla u)dm<\infty\right\}$$ and $W_0^{1,2}(M)$  be the closure of $C_0^\infty(M)$ under the norm $$\|u\|_{W^{1,2}(M)}=\|u\|_{L^2(M)}+\left(\int_M F^2(x,\nabla u)dm\right)^{\frac12}.$$

We remark that  the Finsler-Laplacian is better to  be viewed  in a weak sense due to the lack of regularity, that is, for $u\in W^{1,2}(M)$,
\begin{eqnarray*}
\int_M \phi\Delta_m u dm=-\int_M D\phi(\nabla u) dm \hbox{ for }\phi\in C_0^\infty(M).
\end{eqnarray*}


One can also define a weighted Laplacian on $M$. Given a weakly differentiable function $u$ and a vector field $V$ which does not vanish on $M_u$, the \textit{weighted Laplacian} is defined on the weighted Riemannian manifold $(M,g_V,m)$ by
 \begin{eqnarray*}
\Delta_m^V u:=\div_m(\nabla^V u),
\end{eqnarray*}
where
\begin{equation*}\nabla^V u:=\left\{
\begin{array}{lll} \sum_{i,j=1}^n g^{ij}(x,V)\frac{\partial u}{\partial x_i}\frac{\partial}{\partial x_j}&\hbox{ for }V\in T_x M\setminus \{0\},\\
0 &\hbox{ for }V=0.\\
\end{array}
\right.
\end{equation*}

Similarly, the weighted Laplacian can be viewed in a weak sense. We note that $\Delta_m^{\nabla u} u=\Delta_m u$.

\vspace{2mm}
\subsection{Weighted Ricci curvature and Bochner-Weitzenb\"ock formula}

The Ricci curvature of Finsler manifolds is defined as the trace of the flag curvature.
Explicitly, given two linearly independent vectors $V,W\in T_xM\setminus\{0\}$, the \textit{flag curvature} is defined by
$$\mathcal{K}^V(V,W)=\frac{g_V(R^V(V,W)W,V)}{g_V(V,V)g_V(W,W)-g_V(V,W)^2},$$
where $R^V$ is the \textit{Chern curvature} (or \textit{Riemannian curvature}):
$$R^V(X,Y)Z=D_X^V D_Y^V Z-D_Y^V D_X^V Z-D_{[X,Y]}^V Z.$$
Then the \textit{Ricci curvature} is
$$Ric(V):=F(V)^2\sum_{i=1}^{n-1} \mathcal{K}^V(V,e_i),$$ where $e_1,\cdots,e_{n-1},\frac{V}{F(V)}$ form an orthonormal basis of $T_xM$ with respect to $g_V$.

We recall the definition of  the weighted Ricci curvature on Finsler measure spaces, which was introduced by Ohta in \cite{Oh1},
motivated by the work of Lott-Villani \cite{LV} and Sturm \cite{St,St2} on general metric measure spaces.

\begin{definition}[\cite{Oh1}]
Given a unit vector $V\in T_x M$, let $\eta: [-\varepsilon,\varepsilon]\to M$ be the
geodesic such that $\dot{\eta}(0)=V$. Decompose $m$ as $m=e^{-\Psi} d \hbox{vol}_{\dot{\eta}}$ along $\eta$, where $\hbox{vol}_{\dot{\eta}}$ is the volume form of $g_{\dot{\eta}}$ as a Riemannian metric.
Then
\begin{equation*}
\begin{array}{rcl} \ds\vs 
Ric_n(V)&:=&\left\{
\begin{array}{lll} \ds\vs Ric(V)+(\Psi\circ\eta)''(0) &\hbox{ if }(\Psi\circ \eta)'(0)=0,\\
-\infty &\hbox{ otherwise };
\end{array}\right. \\
\ds\vs Ric_N(V)&:=&Ric(V)+(\Psi\circ\eta)''(0)-\frac{(\Psi\circ\eta)'(0)^2}{N-n},\hbox{ for }N\in(n,\infty);
\\
\ds Ric_\infty(V)&:=& Ric(V)+(\Psi\circ\eta)''(0).
\end{array}
\end{equation*}
For $c\geq 0$ and $N\in [n,\infty]$,  define $$Ric_N(cV):=c^2Ric_N(V).$$
\end{definition}

Ohta proved in \cite{Oh1} that, for $K\in\mathbb{R}$, the condition $Ric_N(V)\geq KF^2(V)$ is equivalent to Lott-Villani and Sturm's \textit{weak curvature-dimension condition} $CD(K,N)$.
We remark that $Ric_n(V)=Ric(V)$ for Berwald manifolds equipped with Busemann-Hausdorff measure (see \cite{Sh} and \cite{Oh1}). In particular, $Ric_n\equiv 0$ for $(\mathbb{R}^n, \|\cdot\|, m_{BH})$ ($m_{BH}$ is just a scalar multiply the Lesbegue measure).

The  following  Bochner-Weitzenb\"ock type formula,  established by Ohta-Sturm  \cite{OS2},
will play a role as the starting point in this paper.
\begin{theorem}[\cite{OS2}, Theorem 3.6]\label{BW formula}
Given $u\in W_{\rm{\rm{loc}}}^{2,2}(M)\bigcap C^1(M)$ with $\Delta_m u\in W_{\rm{\rm{loc}}}^{1,2}(M)$, we have
\begin{eqnarray*} -\int_M D\eta\left(\nabla^{\nabla u}\left(\frac{F^2(x,\nabla u)}{2} \right)\right) dm=
\int_M \eta\bigg\{D(\Delta_m u)(\nabla u)+
 Ric_\infty(\nabla u)+\|\nabla^2 u\|^2_{HS(\nabla u)}\bigg\} dm\end{eqnarray*}
 as well as
\begin{eqnarray*} -\int_M D\eta\left(\nabla^{\nabla u}\left(\frac{F^2(x,\nabla u)}{2} \right)\right) dm\geq \int_M \eta\bigg\{D(\Delta_m u)(\nabla u)+  Ric_N(\nabla u)+\frac{(\Delta_m u)^2}{N}\bigg\} dm\end{eqnarray*}
for any $N\in [n,\infty]$ and all nonnegative functions $\eta\in W_0^{1,2}(M)\bigcap L^\infty(M)$. Here $\|\nabla^2 u\|^2_{HS(\nabla u)}$ denotes the Hilbert-Schmidt norm with respect to $g_{\nabla u}$.
\end{theorem}

\vspace{2mm}
\section{Local Poincar\'e inequality and  Sobolev inequality}

The Sobolev inequality is necessary to run Moser's iteration. In view of the standard theory for general metric measure spaces,  one need a volume doubling condition and a local  uniform Poincar\'e inequality to prove the local Sobolev inequality. From now on, we denote $B_R=B^+_R(p)$ for some $p\in M$ for simplicity.
For $K\geq 0$ and $n\leq N<\infty$, we define
\begin{equation*}
s_{K,N}(t)=\left\{\begin{array}{lll}
\sqrt{\frac{N-1}{K}}\sinh(\sqrt{\frac{K}{N-1}}t), & \hbox{ if }K>0,\\
t, &\hbox{ if }K=0.\end{array}\right.
\end{equation*}

The volume doubling condition follows from the infinitesimal Bishop-Gromov's volume comparison theorem, which was proved by Ohta \cite{Oh1}.
\begin{theorem}[\cite{Oh1}, Theorem 7.3]\label{lem2}
Let $(M, F, m)$ be a complete Finsler manifold  satisfying $Ric_N\geq -K$, $K\geq 0$. For $p\in M$, $v\in T_pM$ with $F(v)=1$, let $\eta:[0,T]\to M$ be the minimal geodesic $\eta(t)=\exp_p(tv)$ for some $T>0$, which does not cross the cut locus of $p$. Suppose $dm(\eta(t))=e^{-V(\dot{\eta}(t))}dvol_{\dot{\eta}(t)}$. 
Then the function $$\frac{e^{-V(\dot{\eta}(t))}t^{n-1} det_{g_{0,t}}(D\exp_p|_{tv})}{s_{K,N}(t)^{N-1}}\hbox{ is nonincreasing},$$
where $det_{g_{0,t}}$ denotes the determinant with respect to $g_{\dot{\eta}(0)}$ and $g_{\dot{\eta}(t)}$ and $D\exp_p|_{tv}$ denotes the differential of the exponential map.
Futhermore, for $R_1\geq R_2$, we have
\begin{eqnarray*}
\frac{m(B_{R_1})}{m(B_{R_2})}\leq \frac{\int_0^{R_1} s_{K,N}^{N-1}(t)dt}{\int_0^{R_2} s_{K,N}^{N-1}(t)dt}\leq e^{2\sqrt{K}R_1}\left(\frac{R_1}{R_2}\right)^N.
\end{eqnarray*}
\end{theorem}

We now prove the local uniform Poincar\'e inequality. The method may be familiar to experts. However, since we have not found this inequality for the Finsler case in any references, for the readers' convenience, we outline the proof here.
\begin{theorem}\label{Poin}Let $(M, F, m)$ be a forward geodesically complete Finsler manifold  satisfying $Ric_N\geq -K$, $K>0$. Then there exist $c=c(N,\l,\L)$ and $C=C(N,\l,\L)$, depending only on $N$, $\l$ and $\L$, such that
\begin{eqnarray}\label{Po}
\int_{B_R} |u-\bar{u}|^2 dm\leq c e^{C\sqrt{K}R}R^2\int_{B_R} F^2(x,\nabla u) dm
\end{eqnarray}
for $B_R\subset M$ and $u\in W_{\rm{\rm{loc}}}^{1,2}(M)$. Here $\bar{u}=\frac{1}{m(B_R)}\int_{B_R} u dm.$
\end{theorem}

To prove Theorem \ref{Poin}, we prove first a slightly weak local Poincar\'e inequality.

\begin{lemma}\label{wPoin}
Let $(M, F, m)$ be a forward geodesically complete Finsler manifold  satisfying $Ric_N\geq -K$, $K\geq 0$. Then there exist $c=c(N,\l,\L)$ and $C=C(N,\l,\L)$, depending only on $N$, $\l$ and $\L$,  such that
\begin{eqnarray}\label{wPo}
\int_{B_R} |u-\bar{u}|^2 dm\leq c e^{C\sqrt{K}R}R^2\int_{B_{(\rho+1)^2R}} F^2(x,\nabla u) dm
\end{eqnarray}
for $B_{(\rho+1)^2R}\subset M$ and $u\in W_{\rm{\rm{loc}}}^{1,2}(M)$, where $\rho$ is the modulus of reversibility \eqref{reverse}.
\end{lemma}

\begin{proof}
For $x,y\in B_R$, 
Let $\gamma_{xy}:[0,d(x,y)]\to M$ be a minimal geodesic from $x$ to $y$ with respect to $F$ with $\dot{\gamma}(0)=v$.  
Notice that $$ -\int_0^{d(y,x)} F(\nabla u(\gamma_{yx}(t)))dt\leq u(y)-u(x)\leq  \int_0^{d(x,y)} F(\nabla u(\gamma_{xy}(t)))dt.$$
Hence we have
\begin{eqnarray}\label{Po1}
&&\int_{B_R} |u-\bar{u}|^2 dm\leq  \frac{1}{m(B_R)} \int_{B_R} \int_{B_R}|u(x)-u(y)|^2 dm(x)dm(y)\nonumber\\
&\leq &\frac{1}{m(B_R)}\int_{B_R}\int_{B_R}\left(\int_0^{d(x,y)} F(\nabla u(\gamma_{xy}(t)))dt+\int_0^{d(y,x)}F(\nabla u(\gamma_{yx}(t)))dt\right)^2 dm(x)dm(y)\nonumber\\
&=&\frac{4}{m(B_R)}\int_{B_R}\int_{B_R}\left(\int_0^{d(x,y)} F(\nabla u(\gamma_{xy}(t)))dt\right)^2 dm(x)dm(y)\nonumber\\
&\leq &\frac{4(\rho+1)R}{m(B_R)}\int_{B_R}\int_{B_R}\int_0^{d(x,y)} F^2(\nabla u(\gamma_{xy}(t))) dt dm(x)dm(y).
\end{eqnarray}
For simplicity of notation, we denote by $d=d(x,y).$ Recall the modulus of reversibility $\rho=\sup_{v\neq 0}\frac{F(v)}{\bar{F}(v)}$. We have \begin{eqnarray}\label{dd}
d\leq d(x,p)+d(p,y)\leq \rho d(p,x)+d(p,y)\leq (\rho+1)R.
\end{eqnarray}

Note that $\bar{\gamma}_{yx}(t):=\gamma_{xy}(d-t)$ is the geodesic from $y$ to $x$ with respect to the reversed norm $\bar{F}$. It is easy to see from definition of $F^*$ that $F^*(Du)=\bar{F}^*(-Du)$, which implies \begin{eqnarray}\label{FF}
F(\nabla u)=\bar{F}(\bar{\nabla}(-u)),
\end{eqnarray}
 where $\bar{\nabla}(-u)$ is the gradient vector of $-u$ with respect to $\bar{F}$.
Thus
\begin{eqnarray}\label{xxx}&&
 \int_{0}^{d/2}F^2(\nabla u(\gamma_{xy}(t)))dt= \int_{d/2}^{d}F^2(\nabla u(\gamma_{xy}(d-t)))dt=\int_{d/2}^{d}\bar{F}^2(\bar{\nabla}(-u)(\bar{\gamma}_{yx}(t)))dt.
\end{eqnarray}

Replacing \eqref{xxx} into \eqref{Po1}, we have
\begin{eqnarray}\label{yyy}
&&\int_{B_R} |u-\bar{u}|^2 dm\leq  \frac{1}{m(B_R)} \int_{B_R} \int_{B_R}|u(x)-u(y)|^2 dm(x)dm(y)\nonumber\\
&\leq &\frac{4(\rho+1)R}{m(B_R)}\int_{B_R}\int_{B_R}\int_{\frac d2}^{d} \left(F^2(\nabla u(\gamma_{xy}(t)))+ \bar{F}^2(\bar{\nabla}(-u)(\bar{\gamma}_{yx}(t)))\right)dt dm(x)dm(y).\end{eqnarray}

For any $z= \gamma_{xy}(t)$, one can decompose \begin{eqnarray}\label{Po2}dm(z)=e^{-V(\dot{\gamma}_{xy}(t))}dvol_{g_{\dot{\gamma}_{xy}(t))}}=e^{-V(\dot{\gamma}_{xy}(t))}det_{g_{0,t}}(D \exp_x|_{tv})t^{n-1}dtd\xi_x,\end{eqnarray}
where $d\xi_x$ denotes the area form of unit sphere $\{V\in T_xM: F(x,V)=1\}$ in $T_xM$.
Now let $\phi_{x,t}: M\to M$ be the map $\phi_{x,t}(y)=\gamma_{xy}(t)=\exp_x(tv)$. 
From Theorem \ref{lem2}, we know that for $t\in[d/2,d],$
\begin{eqnarray}\label{Po3}
&&\frac{e^{-V(\dot{\gamma}_{xy}(t)}t^{n-1} \det_{g_{0,t}}(D\phi_{x,t})}{e^{-V(\dot{\gamma}_{xy}(d)}d^{n-1}\det_{g_{0,d}}(D\phi_{x,d})}\nonumber\\&\geq&\frac{s_{K,N}(t)^{N-1}}{s_{K,N}(d)^{N-1}}\geq \left(\frac{t}{d}\right)^{N-1}e^{-\sqrt{(N-1)K}t}\geq \frac{1}{2^{N-1}}e^{-(\rho+1)\sqrt{(N-1)K}R}.
\end{eqnarray}
The last inequality follows from \eqref{dd}.
It follows from \eqref{Po2} and \eqref{Po3} that
\begin{eqnarray}\label{eqq1}
&&\\
&&\int_{B_R}\int_{B_R}\int_{\frac d2}^{d} F^2(\nabla u(\gamma_{xy}(t)))dt dm(x)dm(y)\nonumber\\&=&\int_{B_R}\int_{B_R}\int_{\frac12d}^{d} F^2(\phi_{x,t}(y),\nabla u(\phi_{x,t}(y)))  dtdm(y)dm(x)\nonumber\\&\leq&c_Ne^{C_{N,\rho}\sqrt{K}R}\int_{B_R}\int_{B_R} \int_{\frac12d}^{d} F^2(\phi_{x,t}(y),\nabla u(\phi_{x,t}(y))) \frac{e^{-V(\dot{\gamma}_{xy}(t))}t^{n-1}\det_{g_{0,t}}(D\phi_{x,t})}{e^{-V(\dot{\gamma}_{xy}(d))}d^{n-1}\det_{g_{0,d}} (D\phi_{x,d})}dtdm(y)dm(x)\nonumber\\
&\leq& c_Ne^{C_{N,\rho}\sqrt{K}R}\int_0^{(\rho+1)R}\int_{B_R}\int_{B_R}F^2(\phi_{x,t}(y),\nabla u(\phi_{x,t}(y))) \frac{e^{-V(\dot{\gamma}_{xy}(t))}t^{n-1}\det_{g_{0,t}}(D\phi_{x,t})}{e^{-V(\dot{\gamma}_{xy}(d))}d^{n-1}\det_{g_{0,d}} (D\phi_{x,d})}dm(y)dm(x)dt\nonumber\\
&=&c_Ne^{C_{N,\rho}\sqrt{K}R}\int_0^{(\rho+1)R}\int_{B_R}\int_{\phi_{x,t}(B_R)}F^2(z,\nabla u(z)) dm(z)dm(x)dt
\nonumber\\&\leq& (\rho+1)c_NRe^{C_{N,\rho}\sqrt{K}R}m(B_R)\int_{B_{(\rho+2)R}} F^2(z,\nabla u(z)) dm(z).\nonumber\end{eqnarray}
Since the Ricci lower bound is common for $F$ and $\bar{F}$, the above computation still holds for $\bar{F}$ and $\bar{\gamma}_{yx}$, i.e.,
we have
\begin{eqnarray}\label{eqq2}
&&\int_{B_R}\int_{B_R}\int_{\frac d2}^{d} \bar{F}^2(\bar{\nabla} (-u)(\bar{\gamma}_{yx}(t)))dt dm(x)dm(y)\\&\leq &(\rho+1)c_NRe^{C_{N,\rho}\sqrt{K}R}m(B_R)\int_{B_{(\rho+1)^2R}} \bar{F}^2(z,\bar{\nabla} (-u)(z)) dm(z).\nonumber\end{eqnarray}
Here $B_{(\rho+2)R}$ is replace by $B_{(\rho+1)^2R}$  since $d(p,z)\leq d(p,x)+d(x,z)\leq R+\rho\bar{d}(x,z)\leq R+\rho(\rho+1)R\leq (\rho+1)^2R$.

Substituting \eqref{eqq1} and \eqref{eqq2} into \eqref{yyy}, and use the fact \eqref{FF}, we conclude that 
\begin{eqnarray*}
\int_{B_R} |u-\bar{u}|^2 dm\leq 8c_N(\rho+1)^2R^2e^{C_{N,\rho}\sqrt{K}R}\int_{B_{(\rho+1)^2R}} F^2(z,\nabla u(z)) dm(z).
\end{eqnarray*}
 We finish the proof of Lemma \ref{wPoin} by virtue of  \eqref{rho}.
\end{proof}

 \noindent\textit{Proof of Theorem \ref{Poin}:} By theWhitney-type covering argument (Corollary 5.3.5
in \cite{SC}), the uniform weak Poincar\'e inequality \eqref{wPo} can be improved in a standard way to the uniform
Poincar\'e inequality \eqref{Po}.\qed

\vspace{1mm}
As long as  the uniform local Poincar\'e inequality, Theorem \ref{Poin} and the
Bishop-Gromov volume comparison theorem, Theorem \ref{lem2}, are available, one can
follow the same argument of Lemma 3.5 by setting $A(R)=\sqrt{K}R$ in
Munteanu-Wang \cite{MW} (see also \cite{HK}) to prove the following
local uniform Sobolev inequality. In fact, their argument only used the structure of metric spaces. Here one only needs  to be careful of the non-revesiblity of $F$ as in the proof of Lemma \ref{wPoin}.

\begin{theorem}\label{Sob}Let $(M, F, m)$ be a forward geodesically complete Finsler manifold  satisfying $Ric_N\geq -K$, $K>0$. Then there exist   constants $\nu>2$ and $C=C(N,\l,\L)$ depending only on $N$, $\l$ and $\L$, such that
\begin{eqnarray}\label{Sobolev0} 
\left(\int_{B_R} |u-\bar{u}|^{\frac{2\nu}{\nu-2}} dm\right)^{\frac{\nu-2}{\nu}} \leq e^{C(1+\sqrt{K}R)}R^2m(B_R)^{-\frac2\nu}\int_{B_R} {F^*}^2(x, Du)dm\end{eqnarray}
for $B_R\subset M$ and $u\in W_{\rm{\rm{loc}}}^{1,2}(M)$.
Consequently, 
\begin{eqnarray}\label{Sobolev} 
\left(\int_{B_R} |u|^{\frac{2\nu}{\nu-2}} dm\right)^{\frac{\nu-2}{\nu}} \leq e^{C(1+\sqrt{K}R)}R^2m(B_R)^{-\frac2\nu}\int_{B_R} {F^*}^2(x, Du)+R^{-2}u^2 dm.\end{eqnarray}
\end{theorem}

\

\section{Proof of Theorem \ref{main thm}}
Let $u$ be a positive harmonic function on $B_{2R}$, $\Delta_m u=0$. It was proved that $u\in W_{\rm{\rm{loc}}}^{2,2}(B_{2R})\bigcap C^{1,\alpha}(B_{2R})$ (see \cite{OS1}). Denote $v=\log u$. One can easily verify that 
\begin{eqnarray}\label{Peq1}
\Delta_m v=-F^2(x,\nabla v).
\end{eqnarray}

Let $f(x)=F^2(x,\nabla v)$. Then $f\in  W_{\rm{\rm{loc}}}^{1,2}(B_{2R})\bigcap C^{\alpha}(B_{2R})$. It follows from the Bochner formula in Theorem \ref{BW formula} and \eqref{Peq1}   that for $0\leq\eta\in W_0^{1,2}(B_{2R})\bigcap L^\infty(B_{2R})$,
\begin{eqnarray}\label{Peq2} \int_{M} D\eta\left(\nabla^{\nabla v}f\right) dm\leq \int_{M} \eta\left(2Df(\nabla v)-2Ric_N(\nabla v)-\frac{2f^2}{N}\right) dm.\end{eqnarray}
Note that, it follows from Lemma 3.5 in \cite{Oh1}  that $\nabla^{\nabla v} f=0\hbox{ a.e. on }  f^{-1}(0)=M\setminus M_v$.  Therefore the LHS of \eqref{Peq2} is actually integrated over $M_v\bigcap B_{2R}$.

Let $\eta=\phi^2f^\beta$, with $\phi\in C_0^\infty({B_{2R}})$, $0\leq\phi\leq 1$, and $\beta>1$. Then $\eta$ is an admissible test function for \eqref{Peq2}. Hence we have from \eqref{Peq2} and $Ric_N\geq -K$ that
\begin{eqnarray}\label{Peq3} 
&&\int_{M_v\bigcap B_{2R}} \beta\phi^2f^{\beta-1}g^{ij}(x,\nabla v)f_if_j+2\phi f^\beta g^{ij}(x,\nabla v)f_i\phi_j dm\nonumber\\&\leq &\int_{B_{2R}} \phi^2f^\beta \left(2Df(\nabla v)+2Kf-\frac{2f^2}{N}\right) dm.
\end{eqnarray}
On $M_v$, we know from \eqref{unif1} and Cauchy-Schwarz inequality that
\begin{eqnarray*}\label{Peq4} 
g^{ij}(x,\nabla v)f_if_j&\geq&\tilde{\l} F(x,\nabla f)^2,\\
|g^{ij}(x,\nabla v)f_i\phi_j|&\leq& \tilde{\Lambda} F(x,\nabla f)F^*(x,D\phi).
\end{eqnarray*}
Hence 
\begin{eqnarray}\label{Peq5} 
\int_{M_v\bigcap B_{2R}} \beta\phi^2f^{\beta-1}g^{ij}(x,\nabla v)f_if_j dm&\geq &\tilde{\l}\int_{M_v} \beta\phi^2 f^{\beta-1}F(x,\nabla f)^2 dm
\nonumber\\&=& \tilde{\l} \int_{{B_{2R}}} \beta\phi^2 f^{\beta-1}F(x,\nabla f)^2 dm,
\end{eqnarray}
\begin{eqnarray}\label{Peq6} 
\int_{M_v\bigcap B_{2R}} -2\phi f^\beta g^{ij}(x,\nabla v)f_i\phi_j dm&\leq& \tilde{\Lambda}\int_{M_v\bigcap B_{2R}} 2\phi f^\beta  F(x,\nabla f)F^*(x,D\phi) dm
\nonumber\\&=&\tilde{\Lambda}\int_{B_{2R}}  2\phi f^\beta  F(x,\nabla f)F^*(x,D\phi) dm.
\end{eqnarray}
Combining \eqref{Peq3} -- \eqref{Peq6}, and taking into account that $$Df(\nabla v)\leq F(x,\nabla v)F(x,\nabla f)=f^\frac12 F(x,\nabla f),$$ we have
\begin{eqnarray*}\label{Peq7} 
\tilde{\l} \frac{4\beta}{(\beta+1)^2}\int_{B_{2R}} \phi^2 F^2(x,\nabla f^{\frac{\beta+1}{2}}) &\leq& \tilde{\Lambda} \frac{4}{\beta+1}\int_{B_{2R}} \phi f^{\frac{\beta+1}{2}}F^*(x,D\phi)  F(x,\nabla f^{\frac{\beta+1}{2}})\nonumber\\&&+\frac{4}{\beta+1}\int_{B_{2R}} \phi^2f^{\frac{\beta+2}{2}}F(x,\nabla f^{\frac{\beta+1}{2}})\nonumber\\&&-\int_{B_{2R}} \frac{2}{N}\phi^2f^{\beta+2}+\int_{B_{2R}} 2K\phi^2f^{\beta+1}.
\end{eqnarray*}
Using H\"older inequality, we obtain
\begin{eqnarray*}
\int_{B_{2R}} \phi^2 F^2(x,\nabla f^{\frac{\beta+1}{2}}) &\leq &C_1\int_{B_{2R}} {F^*}^2(x,D\phi)  f^{\beta+1}+C_2\int_{B_{2R}} \phi^2f^{\beta+2}\nonumber\\&&-C_3 \beta\int_{B_{2R}} \phi^2f^{\beta+2}+C_4\beta K\int_{B_{2R}} \phi^2f^{\beta+1}.
\end{eqnarray*}
We remark that from now on, the constant $C_1, C_2, \cdots,$ depend only on $N,\l,\Lambda$.

Notice that $F^*(x,Df)=F(x,\nabla f)$. For $\beta\geq \frac{2C_2}{C_3}$, we have
\begin{eqnarray}\label{Peq8} 
&&\int_{B_{2R}}  {F^*}^2(x,D(\phi f^{\frac{\beta+1}{2}}))+\frac12 C_3\beta\int_{B_{2R}} \phi^2f^{\beta+2} \nonumber\\&\leq &2C_1\int_{B_{2R}} {F^*}^2(x,D\phi)  f^{\beta+1}+C_4\beta K\int_{B_{2R}} \phi^2f^{\beta+1}.\end{eqnarray}
Using Sobolev inequality \eqref{Sobolev},
we obtain
\begin{eqnarray}\label{Peq15} 
\left(\int_{B_{2R}}\phi^{2\chi} f^{(\beta+1)\chi} \right)^{\frac{1}{\chi}}& \leq& e^{C_5(1+\sqrt{K}R)}R^2m(B_{2R})^{-\frac2\nu}\bigg(C_6\int_{B_{2R}} {F^*}^2(x,D\phi)  f^{\beta+1}\nonumber\\&&+C_7(\beta K+R^{-2})\int_{B_{2R}} \phi^2f^{\beta+1}-\beta\int_{B_{2R}} \phi^2f^{\beta+2}\bigg),\end{eqnarray}
where $\chi=\frac{\nu}{\nu-2}$.


\vspace{2mm}
We first use \eqref{Peq15} to prove the following lemma.
\begin{lemma}\label{lem1}There exists  some large positive constant $C=C(N,\l,\L)$ depending on $N, \l$ and $\L$, such that  for $\beta_0=C(1+\sqrt{K}R)$ and $\beta_1=(\beta_0+1)\chi$,  we have $f\in L^{\beta_1}(B_{\frac32R})$ and \begin{eqnarray}\label{Peq10} 
\|f\|_{L^{\beta_1}\big(B_{\frac32R}\big)}  &\leq &C_8 \frac{(1+\sqrt{K}R)^2}{R^2}m(B_{2R})^{\frac{1}{\beta_1}}.\end{eqnarray}
\end{lemma}

\begin{proof}
Let $C_9\geq \frac{2C_2}{C_3}$ such that $\beta_0=C_9(1+\sqrt{K}R)$ satisfies \eqref{Peq8} and \eqref{Peq15}.
we rewrite \eqref{Peq15} as
\begin{eqnarray}\label{Qeq0} 
\left(\int_{B_{2R}}\phi^{2\chi} f^{(\beta_0+1)\chi} \right)^{\frac{1}{\chi}}& \leq &e^{C_{10}\beta_0}m(B_{2R})^{-\frac2\nu}\bigg(C_6R^2\int_{B_{2R}} {F^*}^2(x,D\phi)f^{\beta_0+1}\nonumber\\&&+C_{11}\beta_0^3\int_{B_{2R}} \phi^2f^{\beta_0+1}-\beta_0R^2\int_{B_{2R}} \phi^2f^{\beta_0+2}\bigg).\end{eqnarray}

 We estimate the second term in RHS of \eqref{Qeq0} as follows,
\begin{eqnarray}\label{Qeq1} C_{11}\beta_0^3\int_{B_{2R}} \phi^2f^{\beta_0+1}&=&
C_{11}\beta_0^3\left(\int_{\{f\geq 2C_{11}\beta_0^2R^{-2}\}} \phi^2f^{\beta_0+1}+\int_{\{f< 2C_{11}\beta_0^2R^{-2}\}} \phi^2f^{\beta_0+1}\right)\nonumber\\
&\leq &\frac12\beta_0R^2 \int_{B_{2R}} \phi^2f^{\beta_0+2}+C_{12}\beta_0^3(\frac{\beta_0}{R})^{2(\beta_0+1)}m(B_{2R}).\end{eqnarray}
For the first term in RHS of \eqref{Qeq0}, we let $\phi=\psi^{{\beta_0+2}}$ with $\psi(x)=\tilde{\psi}(d_p(x))\in C_0^\infty(B_{2R})$ satisfying
\begin{eqnarray*}\label{Peq11} 
0\leq \tilde{\psi}\leq 1,\quad \tilde{\psi}\equiv 1 \hbox{ in } [0,\frac32 R),\quad |\tilde{\psi}'|\leq \frac{C_{13}}{R}.
\end{eqnarray*}
Since $F^*(x,Dd_p)=1$  a.e., $\psi$ satisfies 
\begin{eqnarray*}\label{Peq99} 
0\leq \psi\leq 1,\quad \psi\equiv 1 \hbox{ in } B_{\frac32R},\quad {F^*}(x,D\psi)\leq \frac{C}{R}.\end{eqnarray*}
Hence $R^2{F^*}^2(x,D\phi)\leq C_{14}\beta_0^2\phi^{\frac{2(\beta_0+1)}{\beta_0+2}}$.
Then by H\"older and Young inequalities, 
\begin{eqnarray}\label{Qeq2}
C_6R^2\int_{B_{2R}} {F^*}^2(x,D\phi) f^{\beta_0+1}
&\leq& C_{15}\beta_0^2\int_{B_{2R}} \phi^{\frac{2(\beta_0+1)}{\beta_0+2}}f^{\beta_0+1}\nonumber\\
&\leq &C_{15}\beta_0^2\left(\int_{B_{2R}} \phi^{2}f^{\beta_0+2}\right)^{\frac{\beta_0+1}{\beta_0+2}}m(B_{2R})^{\frac{1}{\beta_0+2}}\nonumber\\
&\leq& \frac12\beta_0R^2\int_{B_{2R}} \phi^{2}f^{\beta_0+2}+C_{16}\beta_0^{\beta_0+3}R^{-2(\beta_0+1)}m(B_{2R}).
\end{eqnarray}
Replacing the estimates \eqref{Qeq1} and \eqref{Qeq2} into \eqref{Qeq0}, we obtain
\begin{eqnarray*}\label{Qeq3} 
\left(\int_{B_{2R}}\phi^{2\chi} f^{(\beta_0+1)\chi} \right)^{\frac{1}{\chi}} \leq e^{C_{10}\beta_0}(C_{12}+C_{16})\beta_0^3(\frac{\beta_0}{R})^{2(\beta_0+1)}m(B_{2R})^{1-\frac2\nu}.\end{eqnarray*}
Taking the $(\beta_0+1)$-th roots on both sides, we get
\begin{eqnarray*}\label{Qeq4} \|f\|_{L^{\beta_1}(B_{\frac32R})} \leq C_{17}(\frac{\beta_0}{R})^{2}m(B_{2R})^{\frac{1}{\beta_1}}.\end{eqnarray*}
\end{proof}

Now we start from \eqref{Peq15} and use the standard Moser iteration to prove Theorem \ref{main thm}.
Let $R_k=R+\frac{R}{2^k}$, $\phi_k\in C_0^\infty(B_{R_k})$ satisfy
\begin{eqnarray*}\label{Peq16} 
0\leq \phi_k\leq 1,\quad \phi_k\equiv 1 \hbox{ in } B_{R_{k+1}},\quad {F^*}(x,D\phi_k)\leq C\frac{2^k}{R}.\end{eqnarray*}
Let $\beta_0,\beta_1$ be the number in Lemma \ref{lem1} and $\beta_{k+1}=\beta_k\chi$ for $k\geq 1$, one can deduce from \eqref{Peq15}  with $\beta+1=\beta_k$ and $\phi=\phi_k$ that (we have dropped the third term in the RHS of \eqref{Peq15} since it is negative)
\begin{eqnarray*}\label{Peq17} \|f\|_{L^{\beta_{k+1}}(B_{R_{k+1}})}\leq e^{C_{18}\frac{\beta_0}{\beta_k}}m(B_{2R})^{-\frac2\nu\frac{1}{\beta_k}}(4^{k}+\beta_0^2\beta_k)^{\frac{1}{\beta_k}}\|f\|_{L^{\beta_{k}}(B_{R_{k}})}.\end{eqnarray*}
Hence by iteration we get
\begin{eqnarray*}\label{Peq18} \|f\|_{L^\infty(B_{R})}\leq  e^{C_{18}\beta_0\sum_k \frac{1}{\beta_k}}m(B_{2R})^{-\frac2\nu \sum_k \frac{1}{\beta_k}}\prod_{k}(4^{k}+\beta_0^3\chi^k)^{\frac{1}{\beta_k}}\|f\|_{L^{\beta_{1}}(B_{\frac32 R})}. \end{eqnarray*}
Since $\sum_k \frac{1}{\beta_k}=\frac{\nu}{2}\frac{1}{\beta_1}$ and $\sum_k \frac{k}{\beta_k}$ converges,
we have
\begin{eqnarray*}\label{Peq19} \|f\|_{L^\infty(B_{R})}&\leq&  C_{19}e^{C_{20}\frac{\beta_0}{\beta_1}}\beta_0^{\frac{3\nu}{2}\frac{1}{\beta_1}}m(B_{2R})^{-\frac{1}{\beta_1}}\|f\|_{L^{\beta_{1}}(B_{\frac32 R})}\nonumber\\&\leq &C_{21}m(B_{2R})^{-\frac{1}{\beta_1}}\|f\|_{L^{\beta_{1}}(B_{\frac32 R})}.\end{eqnarray*}
Using Lemma \ref{lem1}, we conclude
\begin{eqnarray*}\label{Peq20} \|f\|_{L^\infty(B_{R})}\leq C\frac{(1+\sqrt{K}R)^2}{R^2},\end{eqnarray*}
which implies 
\begin{eqnarray*}\label{Peq21} \|F(x,\nabla \log u)\|_{L^\infty(B_{R})}\leq C\frac{1+\sqrt{K}R}{R}.\end{eqnarray*}
For $F(x,\nabla(- \log u))$, the same argument works. Thus we finish the proof of Theorem \ref{main thm}.
\qed

\vspace{2mm}
We now prove Corollary \ref{11}-\ref{13}.

\noindent\textit{Proof of Corollary \ref{11}.} Let $y,z\in B_R$ and $\gamma:[0,1]\to M$ be a minimizing geodesic from $y$ to $z$.
Then \begin{eqnarray*}
\log u(y)-\log u(z)&=&\int_0^1 \frac{d}{dt}\log u(\gamma(t))dt\\&\leq &(\rho+1)R \max_{x\in B_R} F(x,\nabla\log u(x))\leq C(1+\sqrt{K}R).
\end{eqnarray*}
The same holds for $\log u(z)-\log u(y)$. We finish the proof.
\qed

\vspace{2mm}
\noindent\textit{Proof of Corollary \ref{12}.} Letting $R\to \infty$ in Theorem \ref{main thm}, we get that $F(x, \nabla \log u)=F(x, \nabla (-\log u))=0$. Hence $u$ is a constant.
\qed

\vspace{2mm}
\noindent\textit{Proof of Corollary \ref{13}.} Let $v(x)=2\max_{B_{2R}} |u|+u(x)$. $v$ is a positive harmonic function on $(M,F,m)$. Thus
$$\max_{B_R}F(x,\nabla u)=\max_{B_R}F(x,\nabla v)\leq C\frac{2\max_{B_{2R}} |u|+u(x)}{R}\leq C\frac{3\max_{B_{2R}} |u|}{R}.$$
By letting $R\to\infty$, we see from the growth assumption  that the RHS tends to $0$. Hence $u$ is a constant.
\qed

\

\noindent\textbf{Acknowledgments.}  I would like to thank Dr. Bobo Hua for his helpful discussions. I want to express my gratitude to Prof. J\"urgen Jost and Prof. Guofang Wang for their constant encouragement.  I am also grateful to the anonymous referee for his/her careful reading and comments, especially for his correction in the proof of Lemma 3.1 for non-reversible $F$.

\

\end{document}